\documentclass[11pt]{amsart}
\usepackage{amsmath}
\usepackage{cases}
\usepackage{mathrsfs}
\newcommand{\dbar}{\ensuremath{\overline\partial}}

\newcommand{\C}{\ensuremath{\mathbb{C}}}

\makeatletter
\newcommand{\sumprime}{\if@display\sideset{}{'}\sum%
            \else\sum'\fi}
\makeatother

\begin{document}

\numberwithin{equation}{section}

% define theorem environments
\newtheorem{theorem}{Theorem}[section]
\newtheorem{proposition}[theorem]{Proposition}
\newtheorem{conjecture}[theorem]{Conjecture}
\def\theconjecture{\unskip}
\newtheorem{corollary}[theorem]{Corollary}
\newtheorem{lemma}[theorem]{Lemma}
\newtheorem{observation}[theorem]{Observation}
\newtheorem{definition}{Definition}
\numberwithin{definition}{section} %\def\thedefinition{\unskip}
\newtheorem{remark}{Remark}
\def\theremark{\unskip}
\newtheorem{question}{Question}
\def\thequestion{\unskip}
\newtheorem{example}{Example}
\def\theexample{\unskip}
\newtheorem{problem}{Problem}

%\thanks{Research supported by the Key Program of NSFC No. 11171255.}

% \address{Department of Mathematics, Tongji University, Shanghai, 200092, China}
 
% \email{99jujiewu@tongji.edu.cn}
 
 \title[Poincar\'e series and very ampleness]{Poincar\'e series and very ampleness criterion for pluri-canonical bundles}
%\author{Jujie Wu}

\author[Jujie Wu ]{Jujie Wu}
\address{
%School of Mathematics and statistics, Henan University, Kaifeng 475001, Henan Province, People's Republic of China -and- 
	Department of Mathematical Sciences, NTNU, Sentralbygg 2, Alfred Getz vei 1, 7034 Trondheim, Norway}
%\email{jujie.wu@ntnu.no}
\email{99jujiewu@tongji.edu.cn}
\thanks{The author was supported by the Norwegian Research Council grant 240569 and NSFC grant 11601120 .}

\date{}
\maketitle

\bigskip

\begin{abstract}
Let $X$ be a compact quotient of a bounded domain in $\mathbb C^n$. Let $K_X$ be the canonical line bundle of $X$. In this paper, we shall introduce the notion of $S$ very ampleness for the pluri-canonical line bundles $mK_X$ by using the Poincar\'e series. The main result is an effective Seshadri constant criterion of $S$ very ampleness for $mK_X$. An elementary proof of surjectivity of the Poincar\'e map is also given.

\bigskip


\smallskip

\end{abstract}

\tableofcontents

\section{Introduction}

Let $X=\Omega/\Gamma$ be a compact quotient of a bounded domain $\Omega\subset\C^n$. Let $K_X$ be the canonical line bundle of $X$. Denote by $S$ the space of holomorphic sections of $mK_X$ generated by the Poincar\'e series of bounded holomorphic functions on $\Omega$. Let $\sigma_0, \cdots, \sigma_N$ be a basis of $S$. Our motivation comes from the following result of Siegel (see section 40 in \cite{Siegel49}): if $m$ is sufficiently large then there exists $x\in X$ such that
\begin{equation}\label{eq:motivation}
    z \to [\sigma_0(z), \cdots, \sigma_N(z)]
\end{equation}
defines a non-degenerated holomorphic mapping from a neighborhood of $x$ to $\mathbb P^N$. Inspired by Siegel's result, we call $mK_X$ \emph{$S$ very ample} if \eqref{eq:motivation} gives an embedding of $X$ into $\mathbb P^N$. This paper is an attempt to study $S$ very ampleness by using H\"ormander's $L^2$ estimates for the $\dbar$-equation \cite{Hormander65}. Notice that
\begin{equation}\label{eq:equivalent}
    \text{S very ample}\ \Leftrightarrow \ \text{very ample}
\end{equation}
if the Poincar\'e map is surjective, i.e. $S=H^0(X, mK_X)$. By Earle-Resnikoff's theorem (see \cite{Earle69, Resnikoff69}), if $\Omega$ is homogeneous then $H^0(X,mK_X)$ is generated by the Poincar\'e series of some weighted $L^1$ holomorphic functions on $\Omega$ (see \cite{Bers65,Ahlfors64,Bell66} and chapter 7 of Koll\'ar's book \cite{Kollar95} for the background). The proof of Earle-Resnikoff's theorem relies heavily on studies of the bounded symmetric domains (see \cite{Resnikoff69}). Our first main result an elementary proof of Earle-Resnikoff's theorem. In section 3, we shall prove that if $\Omega$ is homogeneous then $H^0(X,mK_X)$ is generated by the Poincar\'e series of polynomials on $\C^n$, in particular, $S=H^0(X, mK_X)$.

By Demailly's result \cite{Demailly96}, lower bound of the Seshadri constant of $K_X$ can be used to study very ampleness of $mK_X$. Thus the effective $S$ very ampleness problem can be solved if $\Omega$ is assumed to be homogeneous. But in general, there exist compact complex surfaces, namely, the Kodaira surfaces $M_{n,m}$, whose universal coverings are bounded non-homogeneous domains in $\C^2$ (see \cite{Kodaira67,Atiyah69,Shabat83}). Thus it is necessary to find a method without using surjectivity of the Poincar\'e map.

Our second main result is in section 4. We shall prove that the Seshadri constant can also be used to give a criterion of $S$ very ampleness \emph{without assuming surjectivity of the Poincar\'e map}. Lower bound estimate of the Seshadri constant of $K_X$ (see \cite{HT00,HT99} for earlier results) is given in the last section. We shall show how to use Berndtsson-Cerd$\grave{a}$'s cut-off function (see \cite{BC95}) to estimate the lower bound of the Seshadri constant.

The main research subject in this paper is compact quotient. In 1971, Griffiths \cite{Griffiths71} showed that for every algebraic variety $V$, there exists a Zariski open subset $U$ of $V$ such that the universal covering of $U$ is a bounded domain of holomorphy in $\C^n$ ($n=\dim V$). Thus it would also be interesting to study quasi-projective quotient. We leave it to the interested reader.

%Finally, we would like to thank Bo-Yong Chen for introducing us this topic and Min Ru for his suggestion to use the Seshadri constant. Part of this work was done while the first author was visiting the Math. Department at Univ. of Houston. She wants to express her heartfelt gratitude to the institution for its invitation.

\section{Basic definitions and results}

Let $\Omega$ be a bounded domain in $\mathbb C^n$. The set ${\rm Iso}(\Omega)$ of biholomorphic mappings from $\Omega$ onto itself has a natural group structure. Let $\Gamma$ be a subgroup of ${\rm Iso}(\Omega)$. Assume that $\Gamma$ is discontinuous, i.e., for every $a\in\Omega$, $\{\gamma(a): \gamma\in \Gamma\}$ is a discrete sequence in $\Omega$. Since $\Omega$ is bounded, by using Cauchy's integral formula, it is easy to see that for every compact subset $K$ in $\Omega$, the number of $\gamma\in \Gamma$ such that $\gamma(K)\cap K=\emptyset$ is finite (see section 36 in \cite{Siegel49}). Thus $\Gamma$ is a countable set. Denote by $j_\gamma$ the complex Jacobian of $\gamma$. Let $m\geq 2$ be an integer. By definition, a (holomorphic) $\Gamma$-automorphic form of weight $m$ on $\Omega$ is a holomorphic function $f: \Omega \rightarrow \C$ such that 
\begin{eqnarray}
f(z) = f(\gamma z ) j^{m}_\gamma(z), \ \ \ \forall \gamma \in \Gamma, z \in \Omega.
\end{eqnarray}	 
The following lemma is essentially known (see section 37 in \cite{Siegel49}) :
\begin{lemma}\label{le:poincare} [Poincar\'e]
Let $\Omega$ be a bounded domain in $\mathbb C^n$. Let $\Gamma$ be a discontinuous subgroup of ${\rm Iso}(\Omega)$. Then
\begin{equation*}
  \sum_{\gamma\in\Gamma} |j_\gamma(z)|^2
\end{equation*}
converges to a smooth function on $\Omega$ locally uniformly.
\end{lemma}
Lemma \ref{le:poincare} implies that if $f$ is a bounded holomorphic function then
\begin{equation}\label{eq:poincare-s}
P_m(f)(z):=\sum_{\gamma\in\Gamma} f(\gamma(z))j^m_\gamma(z)
\end{equation}
is a holomorphic $\Gamma$-automorphic form of weight $m$ on bounded domain $\Omega$. We call $P_m(f)$ the Poinc\'are series of $f$. Denote by $\mathcal{A}(\Gamma,m)$ the space of holomorphic $\Gamma$-automorphic form of weight $m$ on $\Omega$.
 %For each $\gamma_1 \in \Gamma$,
%\begin{eqnarray}
%	P_mf(\gamma_1 z) &=& \sum_\gamma 	P_mf(\gamma \gamma_1 z) j_\gamma^m(\gamma_1 z)\\ \nonumber
%	& =& \sum_\gamma 	P_mf(\gamma \gamma_1 z) j_{\gamma \gamma_1}^m( z) j_{\gamma_1}^{-m}( z)\\ \nonumber
%	& = & j_{\gamma_1}^{-m}( z) P_mf( z)
%\end{eqnarray}

%\begin{definition}
%	A function $f: \Omega \rightarrow \C$ is called a (holomorphic) $\Gamma$-automorphic form of weight $m$ on $\Omega$ if $f$ is holomorphic and 
%	\begin{eqnarray}
%	f(z) = f(\gamma z ) j^{m}_\gamma(z), \ \ \ \forall \gamma \in \Gamma, z \in \Omega.
%	\end{eqnarray}	
%\end{definition}

%\begin{lemma}
%	Let $\Omega$ be a bounded domain in $\mathbb C^n$. Let $\Gamma$ be a discontinuous subgroup of ${\rm Iso}(\Omega)$. $m\geq 2$ be an integer. If $f$ is a bounded holomorphic function on $\Omega$, then
%\end{lemma}
%If $\Omega$ is a bounded domain and $\Gamma$ is a discontinuous subgroup of ${\rm Iso}(\Omega)$. $\Gamma$ is a 
We say that $\Gamma$ is fixed point free if for each $z\in \Omega$, $z= \gamma z $ implies $\gamma =1$. 
%any $\gamma\in\Gamma$ except the identity has no fixed point. 
If $\Gamma$ is discontinuous and fixed point free then $X:=\Omega/\Gamma$ has a canonical structure of a complex manifold induced from that of $\Omega$. We have $\mathcal{A}(\Gamma, m) \simeq H^0(X, mK_X)$, and $\pi^*\sigma = P_m(f)(z) (dz_1\wedge \cdots dz_n)^{\otimes}$, where $K_X$ denotes the canonical line bundle of $X$, $\sigma \in H^0(X, mK_X)$, $P_m(f) \in \mathcal{A}(\Gamma, m) $, and $\pi: \Omega \rightarrow X$ is the universal covering map.
Thus $P_m(f)(dz_1\wedge \cdots dz_n)^{\otimes}$ (\textbf{also denoted by $P_m(f)$} in order for convenience) can be seen as a holomorphic section of $mK_X$ over $X$,
   Siegel (see section 40 in \cite{Siegel49}) proved that if $m$ is large enough then there exist bounded holomorphic functions $f_0, \, \cdots, \ f_n$ and a point $z_0$ in $\Omega$ such that $P_m(f_0)(z_0)\neq 0$ and
\begin{equation}\label{eq:m-analytic-ind}
  \det \left(\frac{\partial(P_m(f_j)/P_m(f_0))}{\partial z_k}\right)_{1\leq j,k\leq n}(z_0)\neq0.
\end{equation}
Notice that \eqref{eq:m-analytic-ind} means that
\begin{equation*}
    z\mapsto  [P_m(f_0)(z), \cdots, P_m(f_n)(z)]\in \mathbb P^n
\end{equation*}
is non-degenerated at $z_0$. Assume further that $X$ is compact without boundary. We know that
\begin{equation*}
    S:=\{P_m(f): f \ \text{is a bounded holomorphic function on} \ \Omega \}
\end{equation*}
is a finite-dimensional subspace of $H^0(X, mK_X)$. Let $\sigma_0, \cdots, \sigma_N$ be a basis of $S$. We say that $mK_X$ is $S$ very ample if $z\mapsto [\sigma_0(z), \cdots, \sigma_N(z)]$ defines an embedding of $X$ into $\mathbb P^N$. The starting point of this paper is to solve the following two problems:

\begin{problem}[Surjectivity of the Poincar\'e map]\label{pr:start1} $S=H^0(X, mK_X)$ ?
\end{problem}

\begin{problem}\label{pr:start2} Can we find an effective bound on $m$ such that $mK_X$ is $S$ very ample ?
\end{problem}

Let's study the first problem first. In section 3, we shall give a simple proof of the known result $S' = H^0(X, mK_X)$ for compact quotients of a bounded symmetric domain, where $S' \supseteq S$ is a set of Poincar\'e series. Thus it suffices to show that $S'=S$, which follows from a standard approximation technique. We suggest the reader to read Chapter 7 of Koll\'ar's book \cite{Kollar95} for the background of the first problem (for earlier results, see \cite{Bers65,Ahlfors64,Bell66,Earle69}).

Let's define $S'$. Let $K_\Omega(z,z)$ be the Bergman kernel of $\Omega$. Put
\begin{equation*}
    ||f||_{p,l}=\int_{\Omega} |f|^p K_\Omega^{-l}.
\end{equation*}
Then
\begin{equation*}
    S':=\{P_m(f): f \ \text{is a holomorphic function on} \ \Omega\ \text{such that} \ \|f\|_{1, (m-2)/2}<\infty  \}
\end{equation*}
By the following lemma (see Proposition 6 in \cite{Bers75}), $S'$ is well defined.

\begin{lemma}\label{le:l1} Let $\Omega$ be a bounded domain in $\mathbb C^n$. Let $\Gamma$ be a subgroup of ${\rm Iso}(\Omega)$. Assume that $\Gamma$ is discontinuous and fixed point free. Let $m\geq 2$ be an integer. Let $f$ be a  holomorphic function on $\Omega$ such that $\|f\|_{1, (m-2)/2}<\infty$. Then $P_m(f)$ converges to a holomorphic function on $\Omega$ such that
\begin{equation*}
  \int_X \|P_m(f)\|K_\Omega^{(2-m)/2} \leq \|f\|_{1, (m-2)/2},
\end{equation*}
where $\int_X$ means integration on a fundamental domain of $X$ and
\begin{equation*}
    \|P_m(f)\|(z):= \sum_{\gamma\in\Gamma} |f(\gamma(z))j^m_\gamma(z)|.
\end{equation*}
\end{lemma}

In \cite{Earle69,Resnikoff69}, the following proposition is proved:

\begin{proposition}\label{pr:l1} Let $\Omega$ be a bounded symmetric domain in $\mathbb C^n$. Let $X:=\Omega/\Gamma$ be a compact quotient of $\Omega$. If $m\geq 2$ is an integer then $H^0(X, mK_X)=S'$.
\end{proposition}

We shall prove that:

\begin{theorem}\label{th:l1} Let $\Omega$ be a bounded symmetric domain in $\mathbb C^n$ with $0\in \Omega$. Let $X:=\Omega/\Gamma$ be a compact quotient of $\Omega$. If $m\geq 2$ is an integer then
\begin{equation*}
    H^0(X, mK_X)= \{P_m(f): f \ \text{is a polynomial on} \ \mathbb C^n  \}.
\end{equation*}
In particular, $H^0(X, mK_X)=S$.
\end{theorem}

Assume the first problem is true. To solve the second, it suffices to find an effective very ampleness criterion for $mK_X$. Let $L$ be an ample line bundle over $X$. Let $h$ be a smooth metric on $L$. Demailly \cite{Demailly92} introduced the Seshadri number $\varepsilon(L,x)$ of $L$ to measure the "local positivity" of $L$ at $x\in X$. By Theorem 7.6 in \cite{Demailly96},
\begin{equation*}
    \varepsilon(L,x)=\sup\{t>0: H(t,x)\neq \emptyset\},
\end{equation*}
where $H(t,x)$ is the space of quasi-plurisubharmonic functions $\phi\in C(X \backslash \{x\})$ such that
\begin{equation*}
    \liminf_{z\to x} \frac{\phi(z)}{\log|z-x|^2}=t,
\end{equation*}
and $i\partial\dbar \phi+i\Theta(L,h) \geq 0$ in the sense of current. Put
\begin{equation*}
    \varepsilon(L)=\inf_{x\in X} \varepsilon(L,x).
\end{equation*}
By using H\"ormander's $L^2$ estimates for the $\dbar$-equation (see \cite{Hormander65}), Demailly proved that (see Proposition 7.10 in \cite{Demailly96})
\begin{equation}\label{eq:De-Seshadri}
    \varepsilon(L)> 2n \Rightarrow K_X+L \ \text{is very ample}.
\end{equation}
By using the Bergman kernel, we know that $K_X$ is ample and
\begin{equation}\label{eq:De-Seshadri}
    (m-1)\varepsilon(K_X)> 2n \Rightarrow mK_X \ \text{is very ample}.
\end{equation}
Thus one may use the lower bound of the Seshadri number of $K_X$ (see \cite{HT99,HT00}) to study the second problem. In general, we don't know whether the first problem is true if $\Omega$ is not assumed to be homogeneous. Since there exist compact complex surfaces, namely, the Kodaira surfaces $M_{n,m}$, whose universal coverings are bounded non-homogeneous domains in $\C^2$ (see \cite{Kodaira67,Atiyah69,Shabat83}), it is necessary to find another method to solve the second problem. We shall use H\"ormander's $L^2$ estimates for the $\dbar$-equation (see \cite{Hormander65}) to prove the following theorem (compare with \eqref{eq:De-Seshadri}):

\begin{theorem}[Main theorem]\label{th:very-ample-Seshadri-Poincare} Let $X$ be a compact quotient of a bounded domain $\Omega\subset\mathbb C^n$. Let $m\geq 2$ be an integer. If
\begin{equation*}
    (m-2)\varepsilon(K_X)>2n
\end{equation*}
then $mK_X$ is $S''$ very ample, where
\begin{equation*}
    S''= \{P_m(f): f \ \text{is a holomorphic function on} \ \Omega\ \text{such that} \ \|f\|_{2, (m-2)}<\infty  \}.
\end{equation*}
\end{theorem}
If the Bergman metric
\begin{equation*}
    \omega:=i\partial\dbar \log K_\Omega
\end{equation*}
satisfies that
\begin{equation*}
    C(\Omega):=\sup_{z\in\Omega}|\dbar\log K_\Omega|^2_\omega(z)<\infty,
\end{equation*}
then one may use Donnely-Fefferman's theorem (see \cite{DF83} or \cite{BernChar00}) to show that
\begin{equation}\label{eq:DF}
    (m-2+C(\Omega)^{-1})\varepsilon(K_X)> 2n \Rightarrow mK_X \ \text{is} \  S''\  \text{very ample}.
\end{equation}
If $\Omega$ is bounded homogeneous, by a result of Kai-Ohsawa \cite{K-O07} (see also \cite{V-G-P63}), one may choose a domain $\Omega'\subset \C^n$, which is biholomorphic equivalent to $\Omega$ such that
\begin{equation*}
|\dbar \log K_{\Omega'}|^2_{i\partial\dbar\log K_{\Omega'}} \equiv \ \text{constant}.
\end{equation*}
More precisely, let
\begin{equation*}
\Omega:=\{(u,v)\in\mathbb C^p \times \mathbb C^q : \  v+\bar v-H(u,u)\in V\}
\end{equation*}
be a Siegel domain of second kind defined by $V$ and $H$, where $V$ is a convex cone in $\mathbb R^q$ containing no entire straight lines and $H$ is $V$-Hermitian, Ishi \cite{Ishi} proved that
\begin{equation*}
C(\Omega)\leq p+2q.
\end{equation*}
Thus if $\Omega$ is bounded symmetric, the above remark also gives an effective $S''$ very ampleness criterion for $2K_X$ (for effective very ampleness criterion of $K_X$ itself, see \cite{Xu} and \cite{Yeung00}).

By a theorem of Hwang and To \cite{HT00}, if $\Omega$ is a bounded symmetric domain then $\varepsilon(K_X,x)$ can be estimated by the injectivity radius
\begin{equation*}
    \rho_x:=\frac 12 \inf_{\gamma\in \Gamma\setminus \{1\}} \rho(x,\gamma x),
\end{equation*}
where $\rho$ is distance function associated to $\omega$. In general, we shall prove:

\begin{proposition}\label{pr:inj-Seshadri} If for every fixed $x\in\Omega$,
\begin{equation}\label{eq:logd}
    \log \rho(x,\cdot) \ \text{is plurisubharmonic on} \ \Omega,
\end{equation}
then
\begin{equation}\label{eq:inj}
    \varepsilon(K_X, \pi(x)) \geq \frac {\rho^2_x}2, \ \forall \ x\in\Omega,
\end{equation}
where $\pi:\Omega\to \Omega/\Gamma=X$ denotes the universal covering mapping.
\end{proposition}

It is known that \eqref{eq:logd} is true if the the sectional curvature of $\omega$ is non-positive (see \cite{GW79}). In particular, every bounded symmetric domain satisfies \eqref{eq:logd}.

In the last section, we shall also use
\begin{equation*}
    D(r,x):=\sup_{z\in\Omega} \frac {\sharp(\Gamma(x)\cap \{\rho(\cdot,z)< r\})}{r^2}
\end{equation*}
to estimate $\varepsilon(K_X)$. We shall prove the following result (generalization of Proposition \ref{pr:inj-Seshadri}):

\begin{proposition}\label{pr:density-Seshadri} If $\Omega$ satisfies \eqref{eq:logd} then
\begin{equation}\label{eq:density-Seshadri}
    \varepsilon(K_X,\pi(x)) \geq \frac 1{2 D(r,x)}, \ \forall \ r>0, \ x\in \Omega.
\end{equation}
where $\pi:\Omega\to \Omega/\Gamma=X$ denotes the universal covering mapping.
\end{proposition}

The proof of Proposition \ref{pr:density-Seshadri} relies heavily on the special cut-off function $u_r$ (see Theorem 1 in \cite{BC95} by B. Berndtsson and J. O. Cerd$\grave{a}$) used in the weighted interpolation problem. $u_r$ also plays an important role in \cite{Mourougane03} and \cite{MT00}. In particular, one may also use the density function
$D_{\Lambda,\kappa}(x)$ (see \cite{Lan67}, \cite{Seip92,Seip93,SeipWallsten92} and \cite{BC95}) to estimate the lower bound of $\varepsilon(K_X,x)$. We leave it to the interested reader.

\section{Surjectivity of the Poincar\'e map}

We shall prove Theorem~\ref{th:l1} in this section. Let $\Omega$ be a bounded symmetric domain in $\C^n$ with $0\in \Omega$, then it is circular and star-shaped with respect to the origin, that is, $tz\in \Omega$ for each $z\in\Omega$ and $t\in\C$ with $|t|\leq 1$ (see \cite{KW65}).

Let $f$ be a holomorphic function on $\Omega$ such that $\|f\|_{1,l}<\infty$ for some $l\geq0$. For every fixed $0 \leq r<1$, $f_r(z):=f(rz)$ is a holomorphic function on a neighborhood of $\overline{\Omega}$. We want to show that
\begin{equation*}
\|f_r-f\|_{1,l}\to 0,
\end{equation*}
as $r\to 1$.

For $w\in \C$, $|w|<1$, let $F(w) = \int_{\Omega} |f(wz)|K_{\Omega}(z,z)^{-l}$, $F$ is defined since $\Omega$ is star-shaped and circular. Also, since on $\Omega$ there exists a complete orthonormal system of complex homogeneous polynomials we know that $ K_{\Omega}(z,z)^{-l}$ is invariant under $z \rightarrow e^{i\theta}z$, $F$ is radial, i.e. $F(w)= F(|w|)$. Thus $F(r)$ is non-increasing function of $r$, $0\leq r < 1$.

Put
\begin{equation*}
    g(t)=K_{\Omega}(tz,tz).
\end{equation*}
We know that $g$ is a subharmonic function on the unit disc.
Since $\Omega$ is circular then $g(te^{i\theta})=g(t)$.
Thus $g(e^r)$, $-\infty<r<0$, is a convex function of $r$. Since $g$ is bounded near the origin we have $g(e^r)$ is an increasing function of $r$. Thus by 
\begin{equation}\label{eq:K(t)}
    K(tz, tz)\leq K(z,z), \ \forall \ z\in\Omega, \ t\in \C  \ \ \text{with} \ \  0<|t|\leq 1.
\end{equation}
Let $f$ be a holomorphic function on $\Omega$ such that $\|f\|_{1,l}<\infty$ for some $l\geq0$. For every fixed $0 \leq r<1$, $f_r(z):=f(rz)$ is a holomorphic function on a neighborhood of $\overline{\Omega}$. We want to show that
\begin{equation*}
    \|f_r-f\|_{1,l}\to 0,
\end{equation*}
as $r\to 1$. Since $\Omega$ is circular, we have
\begin{equation*}
   \|f_r-f\|_{1,l}=\frac{1}{2\pi}\int_0^{2\pi}\left(\int_{\Omega}|f(re^{i\theta}z)-f(e^{i\theta}z)|
    K_{\Omega}(z,z)^{-l}\right) d\theta.
\end{equation*}
For $w\in \C$, $|w|<1$, let $F(w) = \int_{\Omega} |f(wz)|K_{\Omega}(z,z)^{-l}$, $F$ is defined since $\Omega$ is star-shaped and circular. Also, since on $\Omega$ there exists a complete orthonormal system of complex homogeneous polynomials we know that $ K_{\Omega}(z,z)^{-l}$ is invariant under $z \rightarrow e^{i\theta}z$, $F$ is radial, i.e. $F(w)= F(|w|)$. Thus $F(r)$ is non-increasing function of $r$, $0\leq r < 1$.

Since $\int_0^{2\pi}|f(te^{i\theta}z)|d\theta$ is an increasing function of $t$, we have
\begin{equation*}
    \int_0^{2\pi} |f(te^{i\theta}z)-f(e^{i\theta}z)| d\theta \leq 2\int_0^{2\pi}|f(e^{i\theta}z)|d\theta.
\end{equation*}
Thus $\|f^t-f\|_{1,l}\to 0$. Notice that $f^t$ can be written as $\sum f^j$ on $t^{-1} \Omega$, where $f^j$ are homogeneous polynomials of degree $j$. Thus
\begin{equation*}
    \sup_{z\in\Omega} |f^t(z)-\sum_{j=0}^N f^j(z)| \to 0,
\end{equation*}
as $N\to\infty$, which implies that for every sufficiently small $\delta>0$, there exists a polynomial $h$ such that
\begin{equation*}
    \|f-h\|_{1,l}<\delta.
\end{equation*}
Thus Theorem \ref{th:l1} follows from Lemma \ref{le:l1} and Proposition \ref{pr:l1}. Lemma \ref{le:l1} follows easily from the sub-mean value inequality, but the proof of Proposition \ref{pr:l1} is not easy. For reader's convenience, we include a proof here.

\begin{proof}[Proof of Proposition \ref{pr:l1}] Let $\{\sigma_j\}_{0\leq j\leq N}$ be an orthonormal base of $H^0(X, mK_X)$. Put
\begin{equation*}
    h_j(z)(dz_1\wedge  \cdots \wedge dz_n)^{\otimes m}=(\pi^*\sigma_j)(z),
\end{equation*}
where 
%$dz$ is short for $dz^1\wedge \cdots \wedge dz^n$ and 
$\pi: \Omega\to \Omega/\Gamma=X$ is the universal covering mapping. Let $F\subset \Omega$ be a Dirichlet fundamental domain. Let $\chi_F$ be its characteristic function. Denote by $K_m(z,w)$ the reproducing kernel of the space of holomorphic functions on $\Omega$ such that
\begin{equation*}
    \|f\|_{2,m-1}:=\int_\Omega |f(z)|^2 (K_\Omega(z,z))^{1-m} < \infty.
\end{equation*}
Put
\begin{equation}\label{eq:fj}
    f_j(z)=\int_\Omega (\chi_F h_j)(w) K_m(z,w) (K_\Omega(w,w))^{1-m}.
\end{equation}
It suffices to show that $\|f_j\|_{1,(m-2)/2} <\infty$ and
\begin{equation*}
    P_m(f_j)=h_j,
\end{equation*}
for every $0\leq j\leq N$.
\end{proof}

\begin{lemma}\label{le:fj} $\|f_j\|_{1,(m-2)/2} <\infty$.
\end{lemma}

\begin{proof} By definition, we have
\begin{equation}\label{eq:fj1}
    \|f\|_{2,m-1}=\|(f\circ\gamma) j_\gamma^m \|_{2,m-1}, \ \forall \ \gamma\in {\rm Iso}(\Omega).
\end{equation}
Thus
\begin{equation}\label{eq:fj2}
    K_m(\gamma z, \gamma w) j_\gamma(z)^m \overline{j_\gamma(w)}^m = K_m (z,w), \ \forall \ \gamma\in {\rm Iso}(\Omega).
\end{equation}
Put
\begin{equation*}
    A(w)=K_\Omega (w,w)^{-m/2} \int_\Omega |K_m(z,w)|K_\Omega(z,z)^{1-\frac m2}.
\end{equation*}
By \eqref{eq:fj2}, we have
\begin{equation*}
    A(\gamma w)=A(w), \ \forall \ \gamma\in {\rm Iso}(\Omega).
\end{equation*}
Since $\Omega$ is homogeneous, we have that $A(\gamma w)$ is a constant. Let's denote it by $c_m$. Since the closure of $F$ is compact, we have
\begin{equation*}
    \|f_j\|_{1,(m-2)/2} \leq c_m \sup_{w\in F} |h_j(w)| K_\Omega(w,w)^{-m/2}<\infty.
\end{equation*}
The proof of Lemma \ref{le:fj} is complete.
\end{proof}

\begin{lemma}\label{le:pmfj} $P_m(f_j)=h_j$.
\end{lemma}

\begin{proof} By Lemma \ref{le:fj} and Lemma \ref{le:l1}, $P_m(f_j)$ are well defined. By definition, we have
\begin{equation}\label{eq:pmfj1}
    \int_F P_m(f_j) \overline{h_k} K_\Omega^{1-m}= \int_\Omega f_j \overline{h_k} K_\Omega^{1-m}, \ \forall\ 0\leq j,k \leq N.
\end{equation}
We claim that
\begin{equation}\label{eq:pmfj2}
    \int_\Omega K_m(z,w) \overline{h_k(z)} (K_\Omega(z,z))^{1-m}= \overline{h_k(w)}, \ \forall\ 0\leq k \leq N.
\end{equation}
Notice that for every fixed $k$ and $0<t<1$, $h_k(t\cdot)$ is a holomorphic function on $t^{-1}\Omega$. Thus $\|h_k(t\cdot)\|_{2,m-1} < \infty$. By the reproducing property of $K_m$, we have
\begin{equation}\label{eq:pmfj3}
    \int_\Omega K_m(z,w) \overline{h_k(tz)} (K_\Omega(z,z))^{1-m}= \overline{h_k(tw)}, \ \forall\ 0\leq k \leq N.
\end{equation}
By \eqref{eq:K(t)}, for every $0<t<1$,
\begin{equation*}
    \sup_{z\in\Omega} \big| h_k(tz)(K_\Omega(z,z))^{-m/2} \big| \leq
    \sup_{z\in\Omega} \big| h_k(tz)(K_\Omega(tz,tz))^{-m/2} \big|.
\end{equation*}
Notice that, by definition, we have
\begin{equation*}
    \|h_k\|_{\infty}:=\sup_{z\in\Omega} \big| h_k(z)(K_\Omega(z,z))^{-m/2} \big| =
    \sup_{z\in F} \big| h_k(z)(K_\Omega(z,z))^{-m/2} \big| < \infty.
\end{equation*}
Thus the integrand of the left hand side of \eqref{eq:pmfj3} is dominated by
\begin{equation*}
    \|h_k\|_{\infty} |K_m(z,w)| K_\Omega(z,z)^{1-\frac m2},
\end{equation*}
which is integrable on $\Omega$ by the proof of Lemma \ref{le:fj}. Let $t\to 1$, we get \eqref{eq:pmfj2}. Thus
\begin{equation*}
    \int_\Omega f_j \overline{h_k} K_\Omega^{1-m}= \int_F h_j \overline{h_k} K_\Omega^{1-m}.
\end{equation*}
By \eqref{eq:pmfj1}, we have $P_m(f_j)=h_j$.
\end{proof}

\section{Very ampleness of the pluri-canonical line bundles}

We shall prove Theorem~\ref{th:very-ample-Seshadri-Poincare} and \eqref{eq:DF} in this section.

Let $f$ be a holomorphic function on $\Omega$ such that $\|f\|_{2,m-2} <\infty$. Notice that
\begin{equation*}
\|f\|_{2,m-2}=\|f^2\|_{1, m-2},
\end{equation*}
and $m-2=\frac{2m-2}{2}-1$. By Lemma \ref{le:l1}, we have
\begin{equation}\label{eq:f2}
  \int_X \|P_{2m-2}(f^2)\|K_\Omega^{(2-(2m-2))/2} \leq \|f^2\|_{1, m-2}.
\end{equation}
By Schwarz inequality, we have
\begin{equation*}
    \|P_m(f)\|(z)^2 \leq \|P_{2m-2}(f^2)\|(z) \sum |j_\gamma(z)|^2.
\end{equation*}
By \eqref{eq:f2} and Lemma \ref{le:poincare}, we know that $P_m(f)$ is well defined. Denoted by $h$ the Hermitian metric on $K_X$ defined by the Bergman kernel $K_\Omega(z,z)$. By definition, we have
\begin{equation*}
    P^*\left(i\Theta(K_X, h)\right)=i\partial\dbar \log K_{\Omega},
\end{equation*}
where $P: \Omega\to \Omega/\Gamma=X$ is the natural covering mapping. Assume that
\begin{equation}\label{eq:4-1}
    (m-2)\varepsilon(K_X) >2n.
\end{equation}
By definition, for every fixed $x\in X$, there exists a quasi-plurisubharmonic function $\phi^x\in C(X \backslash \{x\})$ such that
\begin{equation}\label{eq:weight}
   (m-2) \liminf_{z\to x} \frac{\phi^x(z)}{\log|z-x|^2} > 2n,
\end{equation}
and $i\partial\dbar \phi^x+i\Theta(K_X,h) \geq 0$ on $X$. For every $x,y\in \Omega$, put
\begin{equation*}
    \psi^x=P^*(\phi^{P(x)})+\log K_{\Omega}, \  \psi^{x,y}= \frac{P^*(\phi^{P(x)})+P^*(\phi^{P(y)})}2+\log K_{\Omega}
\end{equation*}
We know that $\psi^x$ and $\psi^{x,y}$ are plurisubharmonic functions on $\Omega$. We shall use them to construct holomorphic functions that we need.

In order to prove that $mK_X$ is $S''$ very ample, it suffices to prove the following two lemmas:

\begin{lemma}\label{le:jet} For every fixed $x\in\Omega$, there exist holomorphic functions $f_0, \cdots, f_n$ such that $||f_j||_{2,m-2} <\infty$ for every $0\leq j\leq n$, $P_m(f_0)(x)=1$ and
\begin{equation*}
    \frac{\partial(P_m(f_j)/P_m(f_0))}{\partial z_k}(x)=\delta_{jk}, \ \forall \ 1\leq j,k \leq n.
\end{equation*}
\end{lemma}

\begin{lemma}\label{le:2point} For every fixed $x,y\in\Omega$ such that $\Gamma(x)\cap\Gamma(y)=\emptyset$, there exists a holomorphic function $f$ such that $||f||_{2,m-2} <\infty$ and
\begin{equation*}
   P_m(f)(x)=0, P_m(f)(y)=1.
\end{equation*}
\end{lemma}

\begin{proof}[Proof of Lemma \ref{le:jet}] Take a sufficiently small $r>0$ such that
\begin{equation*}
    \gamma(\{|z-x|<r\}) \cap \{|z-x|<r\} = \emptyset, \ \forall \ \gamma\in \Gamma\backslash \{1\}.
\end{equation*}
Choose a smooth function $\kappa$ such that $\kappa\in C^\infty (\mathbb R)$, $\kappa\equiv1$ on $(-\infty, 1/2)$ and   $\kappa\equiv0$ on $(1, \infty)$. Take
\begin{equation*}
\tilde f_0(z)=\kappa \left(\frac{|z-x|}{r}\right), \ \tilde f_j(z)=(z^j-x^j)\kappa \left(\frac{|z-x|}{r}\right), \ \forall \ 1\leq j\leq n.
\end{equation*}
It is easy to see that
\begin{equation*}
    \int_\Omega |\dbar \tilde f_j |^2_{i\partial\dbar |z|^2} e^{-(m-2)\psi^x-|z|^2} <\infty, \
    \forall \ 0\leq j\leq n.
\end{equation*}
Since $\Omega$ is Bergman complete, we know that $\Omega$ is pseudoconvex. By H\"ormander's theorem (in fact, by Demailly's theorem \cite{Demailly82}, it is enough to assume that $\Omega$ is complete K\"ahler), one may solve $\dbar u_j=\dbar \tilde f_j$ such that
\begin{equation}\label{eq:dbar-uj}
    \int_\Omega |u_j |^2  e^{-(m-2)\psi^x-|z|^2} <\infty, \
    \forall \ 0\leq j\leq n.
\end{equation}
By \eqref{eq:weight}, we have
\begin{equation*}
    u_j(\gamma(x))=0, \nabla u_j(\gamma(x))=0, \
    \forall \ 0\leq j\leq n, \ \gamma\in\Gamma.
\end{equation*}
where $\nabla$ denotes the gradient. Now it suffices to choose
\begin{equation*}
    f_j=\tilde f_j-u_j, , \ \forall \ 0\leq j\leq n.
\end{equation*}
The proof is complete.
\end{proof}

\begin{proof}[Proof of Lemma \ref{le:2point}] Take a sufficiently small $r>0$ such that
\begin{equation*}
    \gamma(\{|z-y|<r\}) \cap \{|z-y|<r\} = \emptyset, \ \forall \ \gamma\in \Gamma\backslash \{1\}.
\end{equation*}
and
\begin{equation*}
    \Gamma(x) \cap \{|z-y|<r\} = \emptyset.
\end{equation*}
Take
\begin{equation*}
\tilde f(z)=\kappa \left(\frac{|z-y|}{r}\right).
\end{equation*}
It is easy to see that
\begin{equation*}
    \int_\Omega |\dbar \tilde f |^2_{i\partial\dbar |z|^2} e^{-(m-2)\psi^{x,y}-|z|^2} <\infty, \
    \forall \ 0\leq j\leq n.
\end{equation*}
By H\"ormander's theorem, one may solve $\dbar u=\dbar \tilde f$ such that
\begin{equation}\label{eq:dbar-u}
    \int_\Omega |u|^2  e^{-(m-2)\psi^{x,y}-|z|^2} <\infty.
\end{equation}
By \eqref{eq:weight}, we have
\begin{equation*}
    u(\gamma(x))=u(\gamma(y))=0, \ \forall \ \gamma\in\Gamma.
\end{equation*}
It suffices to choose $f=\tilde f-u$.
\end{proof}

Now the proof of Theorem \ref{th:very-ample-Seshadri-Poincare} is complete.

Let's proof \eqref{eq:DF} now. We shall use the following version of Donnelly-Fefferman theorem (for the proof, see Theorem 2.3 in \cite{BernChar00}).

\begin{lemma}\label{le:DF} Let $\Omega$ be a pseudoconvex domain in $\C^n$. Let $\varphi$ and $\phi$ be two smooth plurisubharmonic functions such that
\begin{equation}\label{eq:dbounded}
    |\dbar \varphi|^2_{i\partial\dbar\varphi} < r<1,
\end{equation}
on $\Omega$. Then for every smooth $\dbar$-closed $(0,1)$-form $v$ with
\begin{equation*}
    I(v):=\int_\Omega |v|^2_{i\partial\dbar(\phi+\varphi)} e^{-\phi+\varphi}<\infty,
\end{equation*}
there exists smooth function $u$ such that $\dbar u=v$ and
\begin{equation*}
    \int_\Omega |u|^2 e^{-\phi+\varphi} \leq  C_r I(v).
\end{equation*}
\end{lemma}

Fix $\varepsilon>0$ and $x\in\Omega$. Put
\begin{equation*}
    \varphi= (\varepsilon+C(\Omega))^{-1}\log K_{\Omega},
\end{equation*}
and
\begin{equation*}
    \phi= \left(m-2+(\varepsilon+C(\Omega))^{-1}\right)\psi^x+|z|^2.
\end{equation*}
Let $\varepsilon \to 0$, by the same argument as in the proof of Theorem \ref{th:very-ample-Seshadri-Poincare}, we know that \eqref{eq:DF} is true.

\section{Lower bound estimate for $\varepsilon(K_X)$}

We shall use Berndtsson-Cerd$\grave{a}$'s cut-off function (see the proof of Theorem 1 in \cite{BC95}) to prove Proposition \ref{pr:inj-Seshadri} and Proposition \ref{pr:density-Seshadri}.

Put $\phi(z)=\log(\rho(z,x)^2/r^2)$, $r>0$. Assume that $i\partial\dbar\phi \geq 0$. Since $\rho$ is the distance function associated to $\omega$, we have $2|\dbar\rho|^2_{\omega}=|d\rho|^2_{\omega}\leq 1$. Thus
\begin{equation}\label{eq:4-1}
    e^{\phi} (i\partial\phi\wedge \dbar\phi) =\frac{4i \partial\rho\wedge \dbar \rho}{r^2}   \leq \frac{2\omega}{r^2}.
\end{equation}
Let $a(t)$ be a $C^{1,1}$ function on $\mathbb R$ such that
\begin{equation}\label{eq:a1}
    a=0 \ \text{on} \ (0,\infty), \ \ \lim_{t\to -\infty}\frac{a(t)}{t}=1.
\end{equation}
Then we have
\begin{equation*}
    i\partial\dbar a(\phi)=a'(\phi)i\partial\dbar \phi+\frac{a''(\phi)}{e^{\phi}}e^{\phi} (i\partial\phi\wedge \dbar\phi).
\end{equation*}
in the sense of distribution, where $a''\in L_{\rm loc}^{\infty}$. Thus if
\begin{equation}\label{eq:a2}
    a'\geq 0 , \  \frac{a''(t)}{e^t}\geq -c \ \text{on} \ \mathbb R, \ c>0,
\end{equation}
then $i\partial\dbar a(\phi)\geq -2c\omega/r^2$.

Assume that $a''(t)=-ce^t$ on $(-\infty,0)$, then
\begin{equation*}
    a(t)=b+t-ce^t,
\end{equation*}
Since $a\in C^{1,1}(\mathbb R)$, we have $a(0)=a'(0)=0$, thus
\begin{equation}\label{eq:a3}
    b=c=1, \ a(t)=1+t-e^t \ \text{on} \ (-\infty,0),
\end{equation}
and
\begin{equation}\label{eq:a4}
    i\partial\dbar a(\phi)\geq -\frac{2\omega}{r^2}
\end{equation}
on $\mathbb R$. Furthermore, by \eqref{eq:a1}, we have
\begin{equation}\label{eq:a5}
    \lim_{z\to x} \frac{a(\phi(z))}{\log|z-x|^2} =1.
\end{equation}
Put
\begin{equation}\label{eq:a6}
    \psi^x(z)=\sum_{\gamma\in\Gamma} a\left(\log\frac{\rho(\gamma(z),x)^2}{r^2}\right)
    = \sum_{\gamma\in\Gamma} a\left(\log\frac{\rho(z,\gamma(x))^2}{r^2}\right).
\end{equation}
By \eqref{eq:a5}, we have
\begin{equation}\label{eq:psix}
    i\partial\dbar \psi^x(z) \geq -2\frac{\sharp\Gamma(x)\cap\{\rho(\cdot,z)<r\}}{r^2} \omega(z).
\end{equation}
By definition of $D(r,x)$, we have
\begin{equation}\label{eq:psix1}
    i\partial\dbar \psi^x \geq -2D(r,x) \omega.
\end{equation}
Notice that $\psi^x$ can be seen as a quasi-plurisubharmonic function on $X$ with isolated singularity at $P(x)$. Thus by \eqref{eq:a5}, \eqref{eq:psix1} and definition of $\varepsilon(K_X, P(x))$, we have
\begin{equation*}
    \varepsilon(K_X, P(x))\geq \frac{1}{2D(r,x)}.
\end{equation*}
The proof of Proposition \ref{pr:density-Seshadri} is complete. Notice that
\begin{equation*}
    D(\rho_x,x)=\rho_x^{-2}.
\end{equation*}
Thus Proposition \ref{pr:inj-Seshadri} follows from Proposition \ref{pr:density-Seshadri}.
\\

\textbf{Acknowledgements} The author would thank Professor Min Ru and Dr. Xu Wang for their helpful comments and suggestions.

\end{document}